\documentclass[12pt]{article}

\usepackage{diagbox}
\usepackage{mathtools}
\usepackage{bbm}
\usepackage{latexsym}
\usepackage{epsfig}
\usepackage{amsmath,amsthm,amssymb,enumerate}

\usepackage[a-1b]{pdfx}
\usepackage{hyperref}
\allowdisplaybreaks
\def\wC{\wh C}
\def\wrho{\wh \r}

\usepackage{color}

  \def\d{\delta} \def\D{\Delta}
    
\def\G{\Gamma}  \def\k{\kappa}
     
  \def\n{\nu} 
\def\r{\rho}  \def\s{\sigma} 
 \def\om{\omega}

\newtheorem{theorem}{Theorem}
\newtheorem{lemma}[theorem]{Lemma}

\newtheorem{claim}{Claim}

\newcommand{\wh}[1]{\widehat{#1}}

\newcommand{\rdown}[1]{{\left\lfloor #1\right \rfloor}}

\newcommand{\brac}[1]{\left(#1\right)}

\newcommand{\bfrac}[2]{\left(\frac{#1}{#2}\right)}

\newcommand{\set}[1]{\left\{#1\right\}}

\def\E{\mathbb{E}}

\def\Pr{\mathbb{P}}

\newcommand{\ignore}[1]{}

\def\cC{{\mathcal C}}

\newcommand{\beq}[2]{\begin{equation}\label{#1}#2\end{equation}}

\newcommand{\RST}{\textbf{RST}}
\newcommand{\RPM}{\textbf{RPM}}
\newcommand{\RHC}{\textbf{RHC}}

\usepackage{tikz}
\usetikzlibrary{decorations.pathmorphing}
\usetikzlibrary{positioning}
\usetikzlibrary{arrows,automata}
\usetikzlibrary{shapes.misc}
\usetikzlibrary{backgrounds}
\usetikzlibrary{arrows,shapes}

\begin{document}
\author{
Deepak Bal\thanks{Department of Mathematics, Montclair State University, Montclair NJ 07043, USA; e-mail: \texttt{deepak.bal@montclair.edu}},
Alan Frieze\thanks{Department of Mathematical Sciences, Carnegie Mellon University, Pittsburgh PA 15213, USA; e-mail: \texttt{frieze@cmu.edu}; research supported in part by NSF grant DMS1952285.}, 
and Pawe\l{} Pra{\l}at\thanks{Department of Mathematics, Toronto Metropolitan University, Toronto, ON, Canada; e-mail: \texttt{pralat@torontomu.ca}; research supported in part by NSERC grant; part of this work was done while the author was visiting the Simons Institute for the Theory of Computing.}
}

\title{Rainbow spanning trees in randomly coloured $G_{k-out}$}
\maketitle

\begin{abstract}
Given a graph $G=(V,E)$ on $n$ vertices and an assignment of colours to its edges, a set of edges $S \subseteq E$ is said to be rainbow if edges from $S$ have pairwise different colours assigned to them. In this paper, we investigate rainbow spanning trees in randomly coloured random $G_{k-out}$ graphs.
\end{abstract}

%%%%%%%%%%%%%%%%%%%%%%%%%%%%
\section{Introduction}
%%%%%%%%%%%%%%%%%%%%%%%%%%%%

Let $G=(V,E)$ be a graph in which the edges are coloured. A set $S\subseteq E$ is said to be {\em rainbow coloured} if each edge of $S$ is in a different colour. There is by now a large body of research on the existence of rainbow structures in randomly coloured random graphs. Let us highlight a few selected results. Frieze and McKay~\cite{FM} and Bal, Bennett, Frieze and Pra\l{}at~\cite{BBFP} studied the existence of rainbow spanning trees in $G_{n,m}$, the classical Erd\H{o}s--R\'enyi random graph process. Cooper and Frieze~\cite{CF}, Frieze and Loh~\cite{FL} and Ferber and Krivelevich~\cite{FK} studied the existence of rainbow Hamilton cycles in $G_{n,m}$. Janson and Wormald~\cite{JW} studied the existence of rainbow Hamilton cycles in random regular graphs. Finally, Bal, Bennett, P\'erez-Gim\'enez and Pra\l{}at~\cite{RGG}, investigated rainbow perfect matchings and Hamilton cycles in random geometric graphs. Of the most popular random graph models, what is missing here is the random multigraph $G_{k-out}$. The aim of this paper is to initiate the study of these problems in the context of a randomly coloured $G_{k-out}$ graphs.

All asymptotics throughout are as $n \to \infty$ (we emphasize that the notations $o(\cdot)$ and $O(\cdot)$ refer to functions of $n$, not necessarily positive, whose growth is bounded). We say that an event in a probability space holds \emph{with high probability} (or \emph{w.h.p.}) if the probability that it holds tends to one as $n\to \infty$. We often write $G_{k-out}$ when we mean a graph drawn from the distribution $G_{k-out}$.

The random graph $G_{k-out}=G_{k-out}(n)$ is defined as follows. It has vertex set $[n]:=\{1, \ldots, n\}$ and each vertex $i \in [n]$ independently chooses $k$ random distinct neighbours from $[n] \setminus \{i\}$, so that each of the $\binom{n-1}{k}$ sets is equally likely to be chosen. It was shown by Fenner and Frieze~\cite{FF} that $G_{k-out}$ is $k$-connected w.h.p.\ for $k\geq 2$. It was shown by Frieze~\cite{F1} that $G_{2-out}$ has a perfect matching w.h.p., and by Bohman and Frieze~\cite{BF} that $G_{3-out}$ is Hamiltonian w.h.p. All of the above results are sharp. For more details we direct the reader to Chapter~18 in~\cite{Frieze_Book}.

We define the randomly coloured graph $G_{k,q}=G_{k,q}(n)$ (not to be confused with $G_{n,m}$) as follows: the underlying graph on $n$ vertices is $G_{k-out}$ and 
(i) there is a set $Q$ of $q$ colours,
(ii) each colour appears $\r:=\rdown{kn/q}$ or $\r+1$ times (there are $kn-q\r$ \emph{popular} colours that appear $\r+1$ times, the remaining colours are \emph{unpopular} and appear $\r$ times; note that if $q$ divides $kn$, then all colours are unpopular),
(iii) $kn$ colours, including repetitions, are randomly assigned to the $kn$ edges of $G_{k-out}$. Finally, let us note that, without loss of generality, we may assume that $q \le kn$. Indeed, if the number of colours is more than $kn$, then some colours are not used at all and the problem is equivalent to the one with $q=kn$.

In this paper we investigate spanning trees. We will prove the following theorem. 
\begin{theorem}\label{th1}
If $k\geq 2$ and $q\geq n-1$, then $G_{k,q}$ has a Rainbow Spanning Tree (\RST) w.h.p.
\end{theorem}
The result is best possible. Trivially, if $q \le n-2$, then there are not enough colours to create a rainbow tree. If $k = 1$, then $G_{k,q}$ is disconnected w.h.p.~\cite{FF}.

%%%%%%%%%%%%%%%%%%%%%%%%%%%%
\section{Preliminaries}
%%%%%%%%%%%%%%%%%%%%%%%%%%%%

%%%%%%%%%%%%%%%%%%%%%%%%%%%%
\subsection{Colour Monotonicity}\label{sec:monotonicity}
%%%%%%%%%%%%%%%%%%%%%%%%%%%%

In our problem, we randomly colour $kn$ edges of $G_{k,q}$ with $q$ colours and we aim to create a rainbow structure. Recall that there are $q_2 = kn-q\r$ popular colours that are present $\r+1$ times and $q_1 = q-q_2$ unpopular colours that are present $\r$ times. 

It is natural to expect that the more colours are available, the easier it is to achieve our goal. We prove this monotonicity property in the following, slightly broader, context. 
Suppose that we are given a finite set $X$ and a set of colours $C$ where $|C|=|X|$. (In our application, $X$ is the set of $kn$ edges of $G_{k,q}$ and $C$ is the set of $kn$ colours: $q$ colours from set $Q$, including repetitions.) We also have two distinct partitions of $C$: $\mathcal{C}=\{C_1, \ldots, C_q\}$ and $\mathcal{\wC}=\{\wC_1, \ldots ,\wC_{q+1}\}$, for some positive integer $q \le |X|-1$. (In our application, each part corresponds to a colour from set $Q$. Partitions $\mathcal{C}$ and $\mathcal{\wC}$ correspond to colourings with $q$ and $q+1$ colours respectively.) Let $\r = \rdown{ |X|/q }$, $\wrho = \rdown{ |X|/(q+1) }$, $q_2 = |X|-q\r$, and $q_1 = q-q_2$. Suppose that $|C_i|=\r$ for $1 \le i \le q_1$ and $|C_i| = \r+1$ for $q_1 + 1 \le i \le q$, that is, there are $q_1$ parts in $\mathcal{C}$ of size $\r$ and $q_2$ parts of size $\r+1$.

We are given a collection of sets $X_1,X_2,\ldots$, each set $X_i$ is a subset of $X$. (In our application, the $X_i$ are the edges of spanning trees, perfect matchings, Hamilton cycles, etc.) Our goal is to create at least one  rainbow set from this collection. Let us consider a random colouring of $X$ via a random bijection from $C$ to $X$. In order to show that the probability that some $X_i$ is rainbow in a random colouring with $q+1$ colours is at least the corresponding probability when elements of $X$ are coloured with $q$ colours, we need to couple the two partitions. In order to do that we need to consider two cases. 

\textbf{Case 1}: $\wrho = \r$. Partition $\mathcal{\wC}$ is obtained from $\mathcal{C}$ by choosing $\r$ parts in $\mathcal{C}$ of size $\r+1$ and replacing them with $\r+1$ parts of size $\r$ (see Figure~\ref{fig:case1}). We couple the two colourings by first randomly mapping the $|C|-\r(\r+1)$ colours from the parts that are the same in both partitions, and conditioning on the result. If some rainbow $X_i$ is created, then it is clearly present in both colourings. Otherwise, it is easy to see that partition $\mathcal{\wC}$ is at least as likely to complete a rainbow colouring. Indeed, suppose that some $X_i$ has $s$ elements that are not coloured yet; we may assume that $1 \le s \le \r$ as, otherwise, such a set cannot be rainbow via the first partition (it could be rainbow via the second partition if $s = \r+1$). The probability that partition $\cC$ completes a rainbow colouring is equal to $\prod_{i=1}^{s-1} \frac { \r (\r+1) - i (\r+1) }{ \r (\r+1) - i}$ that is at most the corresponding probability for partition $\mathcal{\wC}$, namely, $\prod_{i=1}^{s-1} \frac { \r (\r+1) - i \r }{ \r (\r+1) - i}$.

\begin{figure}[h]
\includegraphics[width=0.95\textwidth]{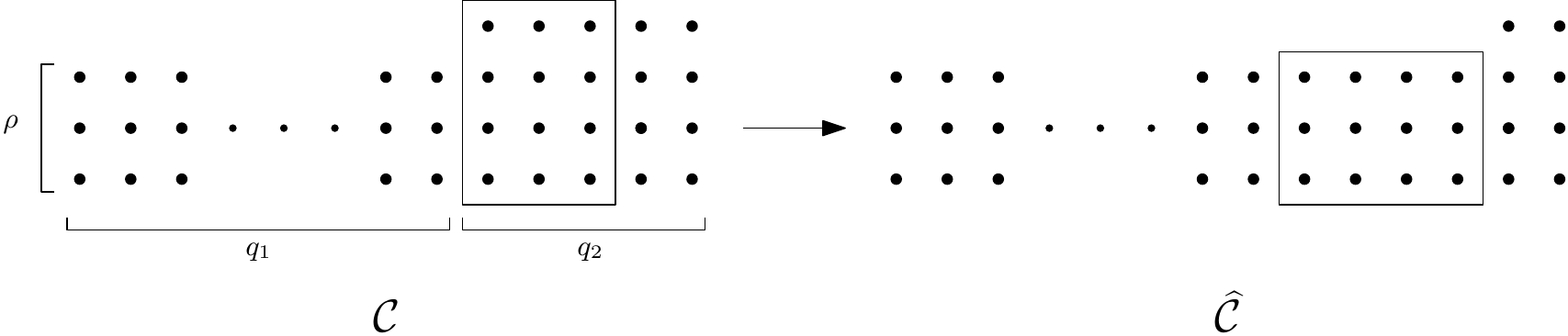}
\caption{The coupling of the two colourings in Case~1.\label{fig:case1}}
\end{figure}

\textbf{Case 2}: $\wrho = \r-1$. As before, we start with partition $\mathcal{C}$ but this time we take all $q_2$ parts of size $\r+1$ (possibly $q_2=0$ so we might not have any) and we choose $\r - 1 - q_2$ parts of size $\r$. To create partition $\wC$, we replace them with $q_2$ parts of size $\r$ and $\r - q_2$ parts of size $\r-1$. As in the other case, either we create some rainbow $X_i$ or partition $\mathcal{\wC}$ is at least as likely to complete a rainbow colouring of some set $X_i$.

The above coupling allows us to concentrate on the minimum number of colours. In particular, in the proof of Theorem~\ref{th1}, without loss of generality, we may assume that $q=n-1$. 

%%%%%%%%%%%%%%%%%%%%%%%%%%%%
\subsection{Degree Monotonicity}\label{sec:degmonotonicity}
%%%%%%%%%%%%%%%%%%%%%%%%%%%%

Based on Section~\ref{sec:monotonicity}, in order to prove Theorem~\ref{th1} we may assume that $q=|Q|=n-1$. In this setup, there are $k$ popular colours present $k+1$ times in $G_{k,q}$ and $q-k=n-1-k$ unpopular colours present $k$ times in $G_{k,q}$. Exactly one out of $k+1$ copies of each popular colour is called \emph{special}. Similarly, there are $k+1$ popular colours present $k+2$ times in $G_{k+1,q}$, $q-(k+1)=n-2-k$ unpopular colours present $k+1$ times, and there are $k+1$ special copies of popular colours (one special copy of each). 

It seems reasonable to expect that if $G_{k,q}$ has a particular rainbow structure w.h.p., then so does $G_{k+1,q}$.  Let $\G_{k+1}$ be the bipartite graph with vertex sets $[n]$ and $Q'=Q\cup \set{q}$, where $q$ is a ``dummy'' vertex that will be associated with special copies of popular colours. For each of the $(k+1)$ edges chosen by $v\in[n]$ in $G_{k+1,q}$, we observe a colour $c$ of that edge without exposing its other endpoint. If a copy of colour $c$ is non-special, then we add an edge between $v$ and $c\in Q$ in $\G_{k+1}$; otherwise, we add an edge between $v$ and the ``dummy'' vertex $q$. Note that we need the extra vertex $q$ to make $\G_{k+1}$ regular. 

Recall that $(k+1)n$ colours, including repetitions, are randomly assigned to the $(k+1)n$ edges of $G_{k+1,q}$. Hence, $\G_{k+1}$ is distributed as a random $(k+1)$-regular bipartite (multi)graph (where multiple edges are allowed but not loops). Indeed, it fits the bipartite configuration model of Bollob\'as~\cite{Bol}. Each vertex could be replaced by a distinct set of $(k+1)$ points, each colour $c \in Q$ naturally corresponds to $(k+1)$ points associated with non-special copies of that colour, and the ``dummy'' vertex $q$ corresponds to $(k+1)$ points associated with special copies of popular colours. Then, we randomly pair these points to get the colouring of the edges of $G_{k+1,q}$. Note that $\G_{k+1}$ contains no information about the actual vertex choices of edges in $G_{k+1,q}$, only their colour. Informally, we can think of each edge of $\G_{k+1}$ being associated with a box containing a random vertex from $[n]$. We do not need to open these boxes for what is next.

We will use the fact that $\G_{k+1}$ is contiguous to the sum of $(k+1)$ independent random perfect matchings---see the survey on random regular graphs by Wormald~\cite{worm}. If we delete one of these matchings, then we obtain a random $k$-regular bipartite graph contiguous to $\G_k$. Arguing as before, we observe that $\G_k$ may be viewed as a random assignment of $kn$ colours, including repetitions, to the $kn$ edges of $G_{k,q}$. ``Opening'' the boxes on each edge of $\G_k$ gives us $G_{k,q}$, which by assumption has a required rainbow structure w.h.p. So, using the above coupling we conclude that $G_{k+1,q}$ also has a rainbow structure w.h.p.

%%%%%%%%%%%%%%%%%%%%%%%%%%%%
\section{Rainbow Spanning Trees}
%%%%%%%%%%%%%%%%%%%%%%%%%%%%

In view of the results in Sections~\ref{sec:monotonicity} and~\ref{sec:degmonotonicity}, we will assume that $k=2$ and $q=n-1$ for this section.

To establish the existence of a \RST\ to prove Theorem~\ref{th1}, we will use the result of Edmonds~\cite{E} on the matroid intersection problem. A finite \emph{matroid} $M$ is a pair $(E,\mathcal{I})$, where $E$ is a finite set (called the \emph{ground set}) and $\mathcal{I}$ is a family of subsets of $E$ (called the \emph{independent sets}) with the following properties:
\begin{itemize}
\item $\emptyset \in \mathcal{I}$,
\item for each $A' \subseteq A \subseteq E$, if $A \in \mathcal{I}$, then $A' \in \mathcal{I}$ (\emph{hereditary property}),
\item if $A$ and $B$ are two independent sets of $\mathcal{I}$ and $A$ has more elements than $B$, then there exists an element in $A$ that when added to $B$ gives a larger independent set than $B$ (\emph{augmentation property}).
\end{itemize}
A maximal independent set (that is, an independent set which becomes dependent on adding any element of $E$) is called a \emph{basis} for the matroid. An observation, directly analogous to the one of bases in linear algebra, is that any two bases of a matroid $M$ have the same number of elements. This number is called the \emph{rank} of $M$. For more details on matroids see, for example,~\cite{Oxley}.

In this scenario, $M_1,M_2$ are matroids over a common ground set $E$ with rank functions $r_1,r_2$, respectively. Edmonds' general theorem shows that 
\begin{equation}
\label{eq:Edmonds}
\max (|I|:I \mbox{ is independent in both matroids}) = \min_{{E_1\cup E_2=E\atop E_1\cap E_2=\emptyset}}(r_1(E_1)+r_2(E_2)),
\end{equation}
where $r_i(E_i)$ is the rank of the matroid induced by $E_i$.

In our application, the common ground set $E$ is the set of coloured multi-edges of $G_{k,q}$. $M_1$ is the cycle matroid of the graph $G_{k-out}$; that is, $S \subseteq E$ is independent in $M_1$ if $S$ induces a graph with no cycle (colours are ignored, two parallel edges are considered to be a cycle of length 2). Hence, for every $S \subseteq E$ we have $r_1(S) = n - \kappa(S)$, where $\kappa(S)$ is the number of components of the graph $G_S=([n],S)$ induced by $S$. $M_2$ is the partition matroid associated with the colours; that is, $S \subseteq E$ is independent in $M_2$ if $S$ has no two edges in the same colour. This time, for every $S \subseteq E$ we have that $r_2(S)$ is the number of distinct colours occurring in $S$. We use Edmonds' theorem to get the following useful observation that has been used a number of times in related contexts. In temporal order, it was used by Fenner and Frieze \cite{FF1}, Frieze and McKay \cite{FM} and by Suzuki \cite{Suz}.
\begin{lemma}
Let $G=(V,E)$ be a multigraph in which each edge is coloured with a colour from a set $Q$. A~necessary and sufficient condition for the existence of a \RST\ is that
\begin{equation}
\label{eq:condition}
\k (C_I)\leq |Q|+1-|I| \hspace{1in} \mbox{for all }I\subseteq Q,
\end{equation}
where $C_I \subseteq E$ is the set of edges of colour from set $I$.
\end{lemma}
\begin{proof}
Clearly, $G$ has a rainbow spanning tree if and only if $G$ contains a set $S$ of coloured edges of size $|V|-1$ such that $S$ is independent both in $M_1$ ($S$ induces a spanning tree) and in $M_2$ ($S$ is rainbow). Since no set of size at least $|V|$ is independent in $M_1$, the necessary and sufficient condition is that the right side of~(\ref{eq:Edmonds}) is at least $|V|-1$. Hence, the desired condition is that for every partition of the edge set $E$ into $E_1$ and $E_2$ we have $r_1 (E_1) + r_2 (E_2) \ge |V|-1$.

Let us fix a partition of $E$ into $E_1$ and $E_2$. Let $J \subseteq Q$ be the set of colours occurring in $E_2$, $E'_2 \subseteq E$ be the set of edges coloured with a colour from $J$, and $E'_1 = E \setminus E'_2$. Clearly, $(E'_1, E'_2)$ is also a partition of $E$, $E_2 \subseteq E'_2$ and so $E'_1 \subseteq E_1$, and $r_2(E_2) = r_2(E'_2) = |J|$. Moreover, since $E'_1 \subseteq E_1$, $r_1(E'_1) \le r_1(E_1)$ and so
$$
r_1 (E_1) + r_2 (E_2) \ge r_1 (E'_1) + r_2 (E'_2).
$$
Therefore, without loss of generality, we may restrict ourselves to sets $E_2$ containing all edges of colour from some set $J \subseteq Q$ and then take $I = Q \setminus J$ (that is, $E_2 = C_J$ and $E_1 = C_I$). The condition to verify is the following:
$$
|V|-1 \le r_1 (E_1) + r_2 (E_2) = (|V| - \kappa(C_I)) + (|Q|-|I|)
$$
which is equivalent to~(\ref{eq:condition}). The proof of the lemma is finished. 
\end{proof}

Recall that $k=2$, $q = n-1$, and so $\rho = k = 2$. For a given $\ell \in [q] \cup \{0\}$, we define the event 
$$
{\cal A}_\ell=\{ \exists I\subseteq Q, |I|=\ell: \k (C_I) \geq q-|I|+2\}.
$$ 
Trivially, ${\cal A}_0$ cannot occur. Since each colour of $Q$ is used at least once, ${\cal A}_1$ cannot occur as well: if $|I|=1$, then $\k (C_I) \le n-1$ but $q-|I|+2 \ge n$. With some additional work we can eliminate ${\cal A}_2$ and ${\cal A}_3$ as well. Note that in $G_{k-out}$, the probability that a vertex $v$ chooses $u$ as a neighbour and vice versa is $O(1/n^2)$. Hence by linearity of expectation, the expected number of multiple edges in $G_{k-out}$ is $O(1)$. The probability that both edges in a multiple edge receive the same (fixed) colour in $G_{k,q}$ is $O(1/q^2)$ and so the expected number of monochromatic multiple edges in $G_{k,q}$ is $O(1/q)=o(1)$. It follows from the first moment method that w.h.p.\ there are no monochromatic multiple edges. Since we aim for a statement that holds w.h.p., we may assume that this property is satisfied. We get that if $|I|=2$, then $\k (C_I) \le n-2$ but $q-|I|+2 \ge n-1$. Similarly, the expected number of triples of colours such that there are two multiple edges that use only these colours is $O(1/q)=o(1)$. Hence, we may assume that for any $I$ of size 3, $|C_I| \ge 5$ and so $\k (C_I) \le n-3$ whereas $q-|I|+2 \ge n-2$. Finally, note that ${\cal A}_{q}$ cannot occur w.h.p.\ since $G_{k-out}$ is connected w.h.p.

We know that if there is no \RST, then ${\cal A}_\ell$ occurs for some $\ell\in [4,q-1]$. Let us concentrate on a minimal $\ell$, corresponding set $I$, and let $S=C_I$. Let us start with the following simple but useful observation.

\begin{claim}
\label{bridges}
$G_S$ has no bridges.
\end{claim}
\begin{proof}
If there is a bridge in $G_S$, then we simply remove it and all edges of the same colour. The number of components increases by at least one and the number of colours decreases by one. Clearly, ${\cal A}_{\ell-1}$ occurs, contradicting the minimality of $\ell$.
\end{proof}

Recall that $\k(S) \ge q-\ell+2 = n-\ell+1$. Suppose then that $G_S$ has $i$ isolated vertices and $n-\ell+x-i$ non-trivial components for some integer $x\geq 1$. Let $T$ be the set of vertices in the non-trivial components of $G_S$. By Claim~\ref{bridges}, $G_S$ has no bridges. Since non-trivial components without bridges have at least three vertices,
\begin{equation}
\label{eq3}
i+3(n-\ell+x-i)\leq n
\end{equation}
which gives us that 
$$
|T| = n - i \le {3 \over 2} (\ell - x) \le {3 \over 2} (\ell - 1).
$$
Let ${\cal B}_\ell$ denote the event 
$$\begin{array}{ll}
\{ \exists I\subseteq Q, |I|=\ell,T\subseteq [n]:& t=|T|\leq 3(\ell-1)/2, \\
	&	\mbox{all edges coloured with $I$ are contained in $T$}, \\ 
	&	\mbox{there are $u\geq \max \{ \rho \ell,t\}$ $I$-coloured edges} \}.
             \end{array} $$
(Here $u\geq t$ because we are dealing with bridgeless components and $u\geq \r\ell$ as each colour appears at least $\r$ times.) Clearly, 
\begin{equation}
\label{eq4}
{\cal A}_\ell \subseteq {\cal B}_\ell \hspace{.5in} \mbox{ for }\ell\geq 4,
\end{equation}
and we will bound the probability of ${\cal B}_\ell$ to deal with small values of $\ell$, that is, for $\ell \le n/20$.

%%%%%%%%%%%%%%%%%%%%%%%%%%%%
\subsection{$4\leq \ell\leq n/20$}
%%%%%%%%%%%%%%%%%%%%%%%%%%%%

Recall that $k=2$, $q=n-1$, and $\r=\rdown{ kn/q } = \rdown{ kn/(n-1) } = 2$. There are $q_1=q-k=n-1-k=n-3$ unpopular colours that appear $k=2$ times and $q_2=k=2$ popular colours that appear $k+1=3$ times.

Note that 
\begin{align*}
\Pr({\cal B}_\ell) & \leq \sum_{t=1}^{3(\ell-1)/2}\binom{n}{t}\sum_{\ell_1+\ell_2=\ell}\binom{q_1}{\ell_1}\binom{q_2}{\ell_2}\binom{tk}{k\ell+\ell_2} \\
& \qquad \times \bfrac{t}{n}^{k\ell+\ell_2} \frac{(k\ell+\ell_2)!}{(kn)(kn-1) \cdots (kn-(k\ell+\ell_2-1))},
\end{align*}
where the second sum is taken over all integers $0 \le \ell_1 \le q_1$, $0 \le \ell_2 \le q_2$ such that $\ell_1+\ell_2=\ell$.
Indeed, in order to estimate $\Pr({\cal B}_\ell)$, we consider all possibilities for $t=|T| \le 3(\ell-1)/2$ (the first sum) and all sets of size $t$ (the term ${n\choose t}$).
We independently consider sets of colours $I \subseteq Q$ with $\ell_1$ unpopular colours and $\ell_2$ popular ones (the second sum). Then, we need to select specific colours (the term $\binom{q_1}{\ell_1}\binom{q_2}{\ell_2}$). Since all edges coloured with $I$ are contained in $T$, we need to select which of the $tk$ edges of $G_{k-out}$ generated by vertices of $T$ are coloured with $I$ (the term $\binom{tk}{k\ell+\ell_2}$). The selected edges need to stay within $T$ (the term $\bfrac{t}{n}^{k\ell+\ell_2}$) and be coloured with $I$ (the last term).

Clearly, $\binom{tk}{k\ell+\ell_2} \le (tk)^{k\ell+\ell_2}/(k\ell+\ell_2)!$. Moreover, since $\ell \le n/20$ and $\ell_2 \le q_2 = k$, 
\begin{eqnarray*}
(kn)(kn-1) \cdots (kn-(k\ell+\ell_2-1)) &\ge& (kn)^{k\ell+\ell_2} \left( 1 - \frac {k\ell+\ell_2}{kn} \right)^{k\ell+\ell_2} \\
&\ge& (kn)^{k\ell+\ell_2} \left( \frac {19}{20} + o(1) \right)^{k\ell+\ell_2} \\
&=& (kn)^{k\ell+\ell_2} \left( \sqrt{\frac {20}{19}} + o(1) \right)^{-2 (k\ell+\ell_2)} \\
&\ge& (kn)^{k\ell+\ell_2} \ 1.03^{-2 (k\ell+\ell_2)}.
\end{eqnarray*}
We get that
\begin{eqnarray*}
\Pr({\cal B}_\ell) &\leq& \sum_{t=1}^{3(\ell-1)/2}{n\choose t}\sum_{\ell_1+\ell_2=\ell}\binom{q_1}{\ell_1}\binom{q_2}{\ell_2} (tk)^{k\ell+\ell_2} \bfrac{t}{n}^{k\ell+\ell_2} \frac{1.03^{\, 2 (k\ell+\ell_2)}}{(kn)^{k\ell+\ell_2}} \\
&=& \sum_{t=1}^{3(\ell-1)/2}{n\choose t}\sum_{\ell_1+\ell_2=\ell}\binom{q_1}{\ell_1}\binom{q_2}{\ell_2} \left( \frac{1.03 t}{n} \right)^{2(k\ell+\ell_2)}. 
\end{eqnarray*}
Since $\binom{q_1}{\ell_1}\binom{q_2}{\ell_2} \le \binom{q}{\ell} \le \binom{n}{\ell} \le (ne/\ell)^{\ell}$, $\binom{n}{t} \le (ne/t)^t$, and $2(k\ell+\ell_2) \ge 2 k \ell = 4 \ell$, we get
\begin{eqnarray*}
\Pr({\cal B}_\ell) &\leq& (\ell + 1) \sum_{t=1}^{3(\ell-1)/2} \left( \frac {ne}{t} \right)^t \left( \frac {ne}{\ell} \right)^{\ell} \left( \frac{1.03 t}{n} \right)^{4 \ell}. 
\end{eqnarray*}
Note that since the ratio between the $(t+1)$-st and the $t$-th term is
$$
\frac {ne}{t+1} \left( \frac {t+1}{t} \right)^{4\ell-t} \ge \frac {ne}{t+1} \ge 2,
$$
the above sum is of the order of its last term. We get that
\begin{eqnarray*}
\Pr({\cal B}_\ell) &=& O \left( \ell \left( \frac {ne}{3\ell/2} \right)^{3\ell/2} \left( \frac {ne}{\ell} \right)^{\ell} \left( \frac{1.03 \cdot 3\ell/2}{n} \right)^{4 \ell} \right) \\
&=& O \left( \ell \left( \frac {2}{3} \cdot e^{5/3} \cdot 1.545^{8/3} \ \frac {\ell}{n} \right)^{3\ell/2} \right) \\
&=& O \left( \ell \left( \frac {12 \, \ell}{n} \right)^{3\ell/2} \right).
\end{eqnarray*}\
Clearly, $\sum_{\ell=4}^{n/20}\Pr({\cal B}_\ell) = O(n^{-6}) = o(1)$, since the sum is dominated by its first term.

%%%%%%%%%%%%%%%%%%%%%%%%%%%%
\subsection{$n/20 < \ell \leq n-\frac{5001 n \log\log n}{\log n}$}
%%%%%%%%%%%%%%%%%%%%%%%%%%%%

We first bound the number of pairwise vertex disjoint cycles in $G_{k-out}$, including loops (cycles of length 1) and parallel edges (cycles of length 2). In our application, we concentrate on $k=2$ but the following bound holds in general.

\begin{lemma}\label{cycles}
Fix any integer $k \ge 2$. W.h.p.\ no family of pairwise vertex disjoint cycles in $G_{k-out}$ consists of more than $3n\log (2k) /\log n$ cycles. 
\end{lemma}
\begin{proof}
Let $\om=\frac{\log n}{2 \log (2k)}$. Let $Z$ denote the number of cycles of length at most $\om$ in $G_{k-out}$ (possibly overlapping). There are $\binom{n}{s}\frac{s!}{2s} = \binom{n}{s}\frac{(s-1)!}{2}$ cycles of length $s$ that one can form out of $n$ labelled vertices. For a given pair of vertices, the probability that there is an edge between them is clearly at most $\frac {2k}{n}$. Hence,
\[
\E\brac{Z}\leq \sum_{s=1}^\om\binom{n}{s}\frac{(s-1)!}{2}\bfrac{2k}{n}^s\leq \sum_{s=1}^\om\frac{(2k)^s}{2s}.
\]
Since the ratio between the two consecutive terms ($(s+1)$st and $s$th) is $(2k)s/(s+1) \ge k \ge 2$,
\[
\E\brac{Z}\leq \frac{(2k)^\om}{\om}= o \left( \exp\left( \frac {\log n}{2 \log (2k)} \log (2k) \right) \right) = o(\sqrt{n}).
\]
The Markov inequality implies that $Z\leq n\log (2k) / \log n$ w.h.p. 

On the other hand, there are trivially at most $n/\om = 2n\log (2k) / \log n$ pairwise vertex disjoint cycles of length greater than $\om$, and the lemma follows. 
\end{proof}

For the next property we need, we assume that $k=2$. As in Section~\ref{sec:degmonotonicity}, let $\G_2$ be the bipartite (multi)graph with vertex sets $[n]$ and $Q' = Q \cup \{q\}$, where $q$ is a ``dummy'' vertex associated with special copies of popular colours. There is an edge $vc$ in $\G_2$ if one of $v$'s {\em choices} has non-special copy of colour $c$; as before, an edge $vq$ occurs in $\G_2$ if one of $v$'s choices is one of the two special copies of popular colours. $\G_2$ is a random 2-regular bipartite graph. Any 2-regular graph is a collection of cycles. We will use a well-known and easy to prove fact that w.h.p.\ there are not too many cycles in $\G_2$. For completeness, we provide an elementary proof.

\begin{lemma}\label{lem:cycles_G2}
$\G_2$ contains at most $(\log n)^2$ cycles w.h.p.
\end{lemma}
\begin{proof}
Let $Z$ denote the number of cycles in $\G_2$. We will use the configuration model of Bollob\'as \cite{Bol} to estimate $\E (Z)$; see also Chapter~11 of Frieze and Karo\'nski~\cite{FK}. Let us select any vertex $v$ and any of the two points in it. We expose the other endpoint of that edge associated with vertex $y$, and move on to the other point associated with $y$. We continue the process until the first cycle is discovered. Then, we select any other point that is not matched yet and continue from there. The probability of closing a cycle at $i$th step of this process is precisely $1/(2n-2i+1)$. Indeed, $2(i-1)$ points are matched before step $i$ and one point is considered at $i$th step, so there are $2n-2i+1$ points available to be matched with the considered point to form an edge; only one of these points (namely, the one associated with vertex $v$ we start with) closes the cycle. We get that
$$
\E (Z) = \sum_{i=1}^n \frac {1}{2n-2i+1} = \sum_{j=1}^{2n} \frac {1}{j} - \sum_{j=1}^{n} \frac {1}{2j} = \ln(2n) - \frac {1}{2} \ln (n) + O(1) = \frac {1}{2} \ln n + O(1).
$$
Showing concentration of $Z$ around its expectation is easy but, since we aim for a slightly weaker result, we may trivially use the Markov inequality to get the desired bound that holds w.h.p.
\end{proof}

We continue concentrating on a minimal $\ell$ (in the range for $\ell$ considered in this section) for which ${\cal A}_\ell$ occurs and the corresponding set $I \subseteq Q$. We let $S=C_I$ and we expose $G_S=([n],S)$, the sub-graph induced by $S$. In particular, $\k(S) \ge n-\ell+1$ and we may assume that $G_S$ is bridgeless by Claim~\ref{bridges}. Let $X_0$ denote the number of isolated vertices of $G_S$. Since $G_S$ is bridgeless, each non-trivial component has a cycle. Hence, the number of non-trivial components, $\eta$, is bounded by the number of pairwise disjoint cycles. By Lemma~\ref{cycles}, we may assume that $\eta \le  \frac{3n\log (2k)}{\log n} \le \frac {5n}{\log n}$. It follows that 
\beq{p1}{
n-\ell+1 \leq \k(S)=X_0+\eta\text{ where }\eta\leq \frac{5n}{\log n}.
}

Almost all sets of colours $I$ yield a graph $G_S$ with many isolated vertices ensuring that~\eqref{p1} does not hold. Only a small fraction of possible sets of colours need special attention. We will use $\G_2$ to estimate how many different configurations we need to take care of. 

Let us concentrate on the set of edges incident with colours $I$ in $\G_2$, ignoring the two edges incident with a ``dummy'' vertex $q$. (Note that removing at most two edges from $G_S$ may only increase the number of isolated vertices, so $X_0 \le X_0'$, where $X_0'$ is the number of isolated vertices in $G_{S'}$; $S'$ is the set obtained from $S$ after removing the edges incident with special copies of popular colours (there are at most two such edges).) By Lemma~\ref{lem:cycles_G2}, this set of edges in $\G_2$ induces a graph $\G$ consisting of $r \leq (\log n)^{2}$ cycles and $s$ paths. The $2s$ endpoints of paths in $\G$ must be in $[n]$ and so each path corresponds to two distinct non-isolated vertices of $G_{S'}$. Let $x$ be the number of vertices in $[n]$ in $\G$ that are of degree 2 (vertices on cycles or internal vertices on paths); they also correspond to non-isolated vertices of $G_{S'}$. The remaining vertices in $\G$ are isolated and the corresponding vertices in $G_{S'}$ \emph{might} be isolated. Note that we did not expose edges yet only assigned colours. It is possible that after exposing edges such vertices will be chosen by other ones and become non-isolated. In any case, $X'_0 \le n - 2s - x$. By comparing the number of edges in $\G$ incident to $[n]$ with the number of edges incident to $I \subseteq Q$, we get $2 \cdot x + 1 \cdot (2s) = 2\ell$ so $x = \ell - s$. Combining the two observations together, we get $X'_0 \le n-2s-x=n-\ell-s$. 

If $s\geq \frac{5n}{\log n}$, then
\[
\k(G_S) = X_0 + \eta \leq X'_0 + \eta \le n -\ell - s + \eta \leq n-\ell-\frac{5n}{\log n}+\eta \le n-\ell,
\]
contradicting~\eqref{p1}.

Let us assume then that $s < \frac{5n}{\log n}$. Let $V_0 \subseteq [n]$ denote the set of vertices in $\G$ that are covered by paths and cycles, and let $V_1=[n]\setminus V_0$. As argued above, $|V_1| = n-\ell-s$ and so $|V_1| \geq \frac{5000 n \log\log n}{\log n}$ since $\ell \le n - \frac{5001 n \log\log n}{\log n}$ and $s < \frac{5n}{\log n} < \frac{n \log\log n}{\log n}$. Vertices in $V_0$ correspond to non-isolated vertices in $G_{S'}$. Vertices in $G_{S'}$ corresponding to vertices in $V_1$ that are isolated in $\G$ do not generate any random edges but it does not mean that they are isolated in $G_{S'}$. Let $V_2 \subseteq V_1$ denote the set of vertices that are incident with an $I$-coloured edge in $\G$ and are {\em chosen by some vertex in $V_0$}. It is important to notice that we have not conditioned on the other endpoints of the choices in~$\G$ (endpoints of random edges generated by vertices in $G_{S'}$ corresponding to vertices in $V_0$ in~$\G$). We may expose these $2 \ell \ge n/10$ edges now, one at a time, each time updating set $V_2$. Provided that $|V_2| \le |V_1| / 50$, we add a new vertex to $V_2$ with probability at least $(49 |V_1|/ 50)/n \ge |V_1|/(2n)$. Let $X \in \text{Bin}\left( n/10, |V_1| / (2n) \right)$ with $\E(X) = |V_1|/20 \ge \frac{250 n \log\log n}{\log n}$. The established coupling implies that
\[
\Pr\brac{ |V_2| \leq \frac{|V_1|}{50}} \leq \Pr\brac{ X \leq \frac{|V_1|}{50}} \le \Pr\brac{ X \leq \frac {\E (X)}{2} },
\] 
and so the Chernoff's bounds imply that
\[
\Pr\brac{|V_2|\leq \frac{ 100 n \log\log n}{\log n}} \leq \exp\left( - \frac {\E (X)}{12}\right) \le \exp\left( - \frac{ 20 n \log\log n}{\log n} \right).
\]
Now, we need to estimate the number of configurations we need to investigate. After orienting (arbitrarily) cycles in $\G_2$, there are at most $\binom{n}{s}$ choices for the beginnings and at most $\binom{n}{s}$ choices for the endings of paths. Such choices yield paths in $\G$. By Lemma~\ref{lem:cycles_G2}, there are at most $2^{(\log n)^2}$ choices to determine which cycles from $\G_2$ should stay in $\G$. Hence, the number of configurations (different bipartite graphs $\G$) we need to deal with is at most
\begin{eqnarray*}
\sum_{s < 5n/\log n} \binom{n}{s}^2 2^{(\log n)^2} &\le& 2 \cdot \binom{n}{5n/\log n}^2 2^{(\log n)^2} \le 2 \cdot \left( \frac{en}{5n/\log n} \right)^{10n/\log n} 2^{(\log n)^2} \\
&\le& \exp \left( \frac{ 10 n \log\log n}{\log n} + O( (\log n)^2 ) \right) \le \exp \left( \frac{ 15 n \log\log n}{\log n} \right).
\end{eqnarray*}
Comparing it with an upper bound for the failure probability for each configuration, we get that if $s < \frac{5n}{\log n}$ and $|V_1| \geq \frac{5000 n \log\log n}{\log n}$, then w.h.p.\ $|V_2| > \frac{100 n \log \log n}{\log n}$. Since $X'_0 \leq n-\ell-s-|V_2| \le n-\ell-|V_2|$, we get that
\[
\k(G_S) \leq X'_0 + \eta \le n-\ell-\frac{100 n \log\log n}{\log n}+\eta < n-\ell,
\]
contradicting~\eqref{p1}.

%%%%%%%%%%%%%%%%%%%%%%%%%%%%
\subsection{$n-\frac{5001 n \log\log n}{\log n}<\ell\leq n-1$}
%%%%%%%%%%%%%%%%%%%%%%%%%%%%

The two special copies of the two popular colours can give us some unnecessary technical problems in the following calculations. So let us delete the two edges, $e_1,e_2$, associated with special copies so that each colour is used exactly twice. The graph $G^*$ obtained this way has $2n-2$ edges. Deleting edges can only increase the number of connected components so it suffices to prove that this quantity is small enough to satisfy~\eqref{eq:condition} after the deletion. It is straightforward to show that $G^*$ remains connected w.h.p.\ but we will not need this fact in our argument. 

For the range of $\ell$ considered in this section, it is easier to concentrate on the largest set of colours $I \subseteq Q$ and the associated set of edges $S=C_I \cup \{e_1, e_2\}$ for which $G_S$ has too many components i.e. $\k(G_S) \ge n-\ell+1$, in violation of \eqref{eq:condition}. Note that for every colour $c$ not in $I$, the two edges of colour $c$ in $G^*$ join distinct components of $G_S$; otherwise, $I\cup\set{c}$ also yields a graph with too many components. This means that there are no edges of colour belonging to $Q \setminus I$ in $G^*$ joining vertices in the same component of $G_S$. Put another way, let $\cC_m$ be the event in $G_{2,q}$ that we can find $m$ colours and the associated $m$ unique pairs of edges $M$ such that 
\begin{itemize}
\item [(i)] $M \cap \{e_1, e_2\} = \emptyset$,
\item [(ii)] each pair in $M$ has the same colour (that is distinct from colours of other pairs in $M$),
\item [(iii)] $G_{\bar M}=G^*-M=G_{2-out}-(M \cup \{e_1, e_2\})$ has at least $m+2$ components, and
\item [(iv)] no edge of $M$ joins two vertices of the same component of $G^*$. 
\end{itemize}
We have to show that $\cC_m$ is unlikely, for $1 \leq m\leq m_0 := \frac{5001 n \log\log n}{\log n}$.

Suppose that $V_1,V_2,\ldots, V_p$ are the components of $G_{\bar M}$. Let $n_i = |V_i|$. We will use $\d_v$ for the number of choices of $v$ in $G_{2-out}$ outside its component in $G_{\bar M}$, and let $\D_i=\sum_{v\in V_i}\d_v$. Note that the number of edges in component $V_i$ is $2n_i - \D_i$. Hence, 
\begin{equation}
\D_i\leq n_i+1, \label{eq:bound_for_delta}
\end{equation} 
as otherwise there are not enough edges inside $V_i$ to get connectivity. It follows that 
\begin{align*}
\Pr(\cC_m) & \leq \binom{n-1}{m} \sum_{p=m+2}^{2m+3} \sum_{\substack{n_1+\cdots+n_p=n\\ n_1\geq n_2\geq \cdots\geq n_p\geq 1}} \frac{1}{\Psi(n_1,\ldots,n_p)}\binom{n}{n_1,n_2,\ldots,n_p} \\
& \qquad \times \sum_{\substack{\D_1+\cdots+\D_p=2m \\ \D_1, \ldots, \D_p \ge 0}} \prod_{i=1}^p\brac{\frac{n_i}{n}}^{2n_i-\D_i}\brac{1-\frac{n_i}{n}}^{\D_i}\binom{2n_i}{\D_i}\frac{1}{\binom{2n}{2m}},
\end{align*}
where
\[
\Psi(n_1,\ldots,n_p) = \prod_{i=1}^n \ell_i! \quad \text{ with } \quad \ell_i=|\set{j \in [p] :n_j=i}|.
\]
Indeed, we first need to choose the colours $M$ which can be done in $\binom{n-1}{m}$ ways. There are at least $m+2$ components in $G_{\bar M}$ but, since we removed exactly $2m+2$ edges from $G_{2-out}$ to get $G_{\bar M}$, the number of them is at most $2m+3$. We then need to choose the component sizes and the components in $\sum_{\substack{n_1+\cdots+n_p=n\\ n_1\geq n_2\geq \cdots\geq n_p\geq 1}} \binom{n}{n_1,n_2,\ldots,n_p} / \Psi(n_1,\ldots,n_p)$ ways. Function $\Psi(n_1,\ldots,n_p)$ removes an implicit ordering of the components in the multinomial coefficient. We then consider all possibilities for the number of $M$ coloured edges (the $2m$ edges present in $G^*$) leaving each component in $\sum_{\substack{\D_1+\cdots+\D_p=2m \\ \D_1, \ldots, \D_p \ge 0}}$ ways. The factor $\binom{2n_i}{\D_i}$ accounts for choosing which of the $2n_i$ edges generated by vertices in $V_i$ have colour in $M$ and are not one of the two special edges, $e_1,e_2$.  The factor $n_i/n$ (respectively, $1-n_i/n$) is the probability that the edge choice of a vertex in $V_i$ is in $V_i$ (respectively, not in $V_i$). Finally, we need to make sure that the $2m$ edges we identified received precisely the $2m$ non-special copies of the $m$ colours we selected. This happens with probability $1/\binom{2n}{2m}$ as any set of $2m$ colours from the set of $2n$ available colours (including repetitions and including special copies) is assigned to these edges with uniform probability; only one of them has colours that are exactly the ones we selected at the very beginning (note that special colours were excluded then). 

Continuing,
\begin{align}
\Pr(\cC_m)&\leq \frac{\binom{n}{m}}{\binom{2n}{2m}}\sum_{p=m+2}^{2m+3}\sum_{\substack{n_1+\cdots+n_p=n\\ n_1\geq n_2 \geq \cdots\geq n_p\geq 1}} \frac{1}{\Psi(n_1,\ldots,n_p)} \binom{n}{n_1,n_2,\ldots,n_p} \nonumber\\
& \qquad \times \sum_{\D_1+\cdots+\D_p=2m}\prod_{i=1}^p\brac{\frac{n_i}{n}}^{2n_i-\D_i}\brac{1-\frac{n_i}{n}}^{\D_i}\binom{2n_i}{\D_i}\nonumber\\
&\le \frac{\binom{n}{m}}{\binom{2n}{2m}}\sum_{p=m+2}^{2m+3}\sum_{\substack{n_1+\cdots+n_p=n\\ n_1\geq n_2 \geq \cdots\geq n_p\geq 1}} \frac{1}{\Psi(n_1,\ldots,n_p)} \binom{n}{n_1,n_2,\ldots,n_p} \nonumber\\
& \qquad \times \sum_{\D_1+\cdots+\D_p=2m}\prod_{i=1}^p\brac{\frac{n_i}{n}}^{2n_i-\D_i}\binom{2n_i}{\D_i}\nonumber\\
&\leq\frac{\binom{n}{m}}{\binom{2n}{2m}}\frac{n!}{n^{2n-2m}}\sum_{\substack{p\geq m+2\\n_1+\cdots+n_{p}=n\\n_1\geq n_2\geq \cdots\geq n_{p}\geq 1\\\D_1+\cdots+\D_{p}=2m}} \frac{1}{\Psi(n_1,\ldots,n_p)} \prod_{i=1}^p\frac{ n_i^{2n_i-\D_i}}{n_i!}\binom{2n_i}{\D_i}.\label{ss0}
\end{align}

Let us prove the following simple structural property of $G_{2-out}$. It will imply that the largest component (of size $n_1$) has size asymptotic to $n$.

\begin{lemma}\label{big}
For $S\subseteq [n]$, let $e^+(S)$ denote the number of choices by vertices in $S$ that are not in $S$, and let $e(S,[n] \setminus S) = e^+(S) + e^+([n] \setminus S)$ denote the number of edges in $G_{2-out}$ that are between $S$ and its complement. Then, w.h.p.\ the following property holds for any $m$ such that $1 \le m \le m_0 := \frac{5001 n \log\log n}{\log n}$:
$$
\text{for all $S\subseteq [n], 9m \leq |S| \leq n/2$, we have $e(S,[n] \setminus S) > 2m+2$.}
$$
\end{lemma}
\begin{proof}
We will independently deal with small and large sets by proving the following statement. For any $m$ such that $1 \le m \le m_0 := \frac{5001 n \log\log n}{\log n} = o(n)$, the following two properties hold with probability $1-O(n^{-1})$:
\begin{enumerate}[(a)]
\item 
\beq{badS}{
\text{for all $S\subseteq [n], 9m \leq |S|=s\leq n/200$, we have $e^+(S) \geq s/2 > 2m+2$.}
}
\item 
\beq{badS1}{
\text{for all $S\subseteq [n], n/200\leq |S|=s\leq n/2$, we have $e(S,[n] \setminus S) > 2m+2$.}
}
\end{enumerate}

Note that, since $16 e^2 < 200$,
\begin{align*}
\Pr(\neg\eqref{badS}) &\leq \sum_{s=9m}^{n/200} \binom{n}{s} \binom{2s}{3s/2} \bfrac{s}{n}^{3s/2} \leq \sum_{s=9m}^{n/200} \bfrac{en}{s}^s 2^{2s} \bfrac{s}{n}^{3s/2} \\
&\leq \sum_{s=9m}^{n/200} \brac{\frac{s}{n}\cdot 16 e^2}^{s/2}=O(n^{-1}),
\end{align*}
so property~(a) holds. 

To see that property~(b) holds too, note that
\begin{align*}
\Pr(\neg\eqref{badS1}) &\leq \sum_{s=n/200}^{n/2} \sum_{i=0}^{2m+2} \sum_{j=0}^{i} \binom{n}{s} \binom{2s}{j} \bfrac{s}{n}^{2s-j} \brac{1-\frac{s}{n}}^{j} \binom{2(n-s)}{i-j} \bfrac{s}{n}^{i-j} \brac{1-\frac{s}{n}}^{2(n-s)-i+j} \\
&=\sum_{s=n/200}^{n/2} \sum_{i=0}^{2m+2} \sum_{j=0}^{i} \binom{n}{s} \binom{2s}{j} \binom{2(n-s)}{i-j} \bfrac{s}{n}^{2s+i-2j} \brac{1-\frac{s}{n}}^{2(n-s)-i+2j} \\
&= O(m^2) \sum_{s=\s n=n/200}^{n/2} \bfrac{1}{\s^\s(1-\s)^{1-\s}}^n \binom{2n}{2m+2}^2 \s^{2\s n - (2m+2)} (1-\s)^{2(1-\s)n - (2m+2)} \\
&= O(m^2) \sum_{s=\s n=n/200}^{n/2}(\s^\s(1-\s)^{1-\s})^n \binom{2n}{2m+2}^2 \left( \s(1-\s) \right)^{-(2m+2)} \\
&= O(m^2n) \ c^n \exp \left( O \left( \frac {n (\log \log n)^2}{\log n} \right) \right) = c^n e^{o(n)} = O(n^{-1}),
\end{align*}
where $c=(1/200)^{1/200}(199/200)^{199/200}<0.97$. (In the above computation, $i$ corresponds to $e(S,[n] \setminus S)$ and $j$ corresponds to $e^+(S)$.)
\end{proof}

The lemma implies that we may assume that 
\beq{n1}{
n_1\geq n-9m \text{ in \eqref{ss0}};
} 
otherwise, the number of edges joining distinct components in $G_{\bar M}$ would be greater than $2m+2$. Hence, we may rewrite~\eqref{ss0} as follows:
\begin{align*}
\Pr(\cC_m) &\leq \frac{\binom{n}{m}}{\binom{2n}{2m}}\frac{n!}{n^{2n-2m}} \sum_{p=m+2}^{2m+3} \sum_{s=p-1}^{9m} \sum_{\substack{n_1+\cdots+n_{p}=n\\n-s=n_1\geq n_2\geq \cdots\geq n_{p}\geq 1\\\D_1+\cdots+\D_{p}=2m}} \frac{1}{\Psi(n_1,\ldots,n_p)} \prod_{i=1}^p\frac{ n_i^{2n_i-\D_i}}{n_i!}\binom{2n_i}{\D_i}.
\end{align*}
Define 
\[
f(a,x)=a^{2a-x}\binom{2a}{x}.
\]
Note that if $x\geq1$, then
$$
\frac{f(a,x)}{f(a,x-1)}=\frac{2a-x+1}{ax}\leq \frac{2}{x},
$$
and so $f(a,x) \le \frac {2^x}{x!} f(a,0)$. Using this observation, we get that
\begin{align*}
\sum_{\D_1+\cdots+\D_p = 2m} & \prod_{i=1}^{p}n_i^{2n_i-\D_i}\binom{2n_i}{\D_i} \\
&\le \sum_{\D_1, \ldots, \D_p \leq 2m}\prod_{i=1}^{p}n_i^{2n_i-\D_i}\binom{2n_i}{\D_i}\\
&= \sum_{\D_1, \ldots, \D_{p-1} \leq 2m} \brac{ \prod_{i=1}^{p-1}n_i^{2n_i-\D_i} \binom{2n_i}{\D_i} } \sum_{\D_p\leq 2m}n_p^{2n_p-\D_p} \binom{2n_p}{\D_p} \\
&\leq \sum_{\D_1, \ldots, \D_{p-1} \leq 2m} \brac{ \prod_{i=1}^{p-1}n_i^{2n_i-\D_i} \binom{2n_i}{\D_i} } n_p^{2n_p-0} \binom{2n_i}{0} \brac{1+\frac{2}{1}+\frac21\cdot\frac{2}{2}+\frac21\cdot\frac22\cdot\frac23+\cdots} \\
&\leq \sum_{\D_1, \ldots, \D_{p-1} \leq 2m} \brac{ \prod_{i=1}^{p-1}n_i^{2n_i-\D_i} \binom{2n_i}{\D_i} } n_p^{2n_p} \sum_{k \ge 0} \frac {2^k}{k!} \\
&= e^2 n_p^{2n_p} \sum_{\D_1, \ldots, \D_{p-1} \leq 2m} \prod_{i=1}^{p-1}n_i^{2n_i-\D_i}\binom{2n_i}{\D_i}\\
&\leq \ldots \le (e^2)^p \prod_{i=1}^{p}n_i^{2n_i}.
\end{align*}
Unfortunately, the constant
$$
e^2 = \sum_{k \ge 0} \frac {2^k}{k!}
$$
associated with the sum over all possible values of $\Delta_i$ is too large for the final argument to follow. Fortunately, any constant smaller than $e^2$ would work. We may squeeze a bit more by using the following observation. Since $\sum_{i=2}^p n_i = s \le 9m$ (see \eqref{n1}) and $p-1 \ge m+1$, there are at most $p/2$ values of $n_i$, $i \ge 2$, that are at least $18$ ($n_1 \sim n$ certainly is at least 18); the remaining ones are at most 17. It is important to notice that the sequence of $n_i$'s is non-increasing so we conclude that at least the last $p/2-1$ values of $n_i$ are at most 17. As a result, since $\Delta_i \le n_i +1$ (see~(\ref{eq:bound_for_delta})), the corresponding values of $\Delta_i$'s are at most 18 (we will refer to them as small). The contribution from small $\D_i$'s is $A$, where
$$
A = \sum_{k \le 18} \frac {2^k}{k!} \le e^2 - 10^{-12}.
$$
We get that
\begin{align*}
\sum_{\D_1+\cdots+\D_p = 2m} \prod_{i=1}^{p}n_i^{2n_i-\D_i}\binom{2n_i}{\D_i} & \le (e^2)^{p/2+1} A^{p/2-1} \prod_{i=1}^{p}n_i^{2n_i} \\
& = \frac {e^2}{A} ( e^2 A)^{p/2} \prod_{i=1}^{p}n_i^{2n_i} \le 2 B^p \prod_{i=1}^{p}n_i^{2n_i},
\end{align*}
where
$$
B = e \sqrt{A} \le e^2 - 10^{-12}.
$$

It follows that
$$
\Pr(\cC_m) \leq 2 \frac{\binom{n}{m}}{\binom{2n}{2m}}\frac{n!}{n^{2n-2m}} \sum_{p=m+2}^{2m+3} B^{p} \sum_{s=p-1}^{9m} \sum_{\substack{n_1+\cdots+n_{p}=n\\n-s=n_1\geq n_2\geq \cdots\geq n_{p}\geq 1}} \frac{1}{\Psi(n_1,\ldots,n_p)} \prod_{i=1}^p\frac{ n_i^{2n_i}}{n_i!}.
$$
Now, for a given sequence $n_1, \ldots, n_p$ and an integer $i$, $2 \le i \le p$, let us define
\[
g(a,b)=\frac{a^{2a}b^{2b}}{a!b! \Psi(a, n_2, \ldots, n_{i-1}, b, n_{i+1}, \ldots, n_p)} 
\]
and suppose that $n \sim a \gg 9m \ge b>1$. Then, since $(1+ 1/x)^x$ is an increasing function of $x$ (tending to $e$ but we do not need this fact) and 
$$
\frac {\Psi(a, n_2, \ldots, n_{i-1}, b, n_{i+1}, \ldots, n_p)}{\Psi(a+1, n_2, \ldots, n_{i-1}, b-1, n_{i+1}, \ldots, n_p)} \le 2,
$$
we get that 
\begin{align*}
\frac{g(a,b)}{g(a+1,b-1)} &\le \frac{a^{2a}}{(a+1)^{2a+2}} \cdot \frac{b^{2b}}{(b-1)^{2b-2}} \cdot \frac {a+1}{b} \cdot 2\\
&= \frac { \left( 1 + \frac {1}{b-1} \right)^{2(b-1)} } {  \left( 1 + \frac {1}{a} \right)^{2a} } \cdot \frac {b}{a+1} \cdot 2 \leq \frac{2b}{a} \le \frac {20m}{n}. 
\end{align*}
It implies that the terms corresponding to sequences of $n_i$'s with larger values of $n_1$ (smaller values of $s$ in our bound) are much larger. On the other hand, there are more sequences with smaller values of $n_1$ (larger values of $s$) to consider. However, since $n_2+\cdots+n_{p} = s$ and $n_i \ge 1$, there are only $\binom{s-1}{p-2}$ choices for the sequence of $n_i$'s to consider. Combining the two observations together, we get that the term corresponding to $s=p-1$ (and the associated unique sequence of $n_i$'s) is a dominating term:
\begin{align*}
\Pr(\cC_m) &\leq 2 \frac{\binom{n}{m}}{\binom{2n}{2m}}\frac{n!}{n^{2n-2m}} \sum_{p=m+2}^{2m+3} B^{p} \frac{ (n-p+1)^{2(n-p+1)}}{(n-p+1)!} \frac {1}{\Psi(n-p+1,1,\ldots,1)} \\
& \qquad \cdot \left( 1 + \sum_{s=p}^{9m} \binom{s-1}{p-2} \left( \frac {20m}{n} \right)^{s-(p-1)} \right) \\
&\leq 3 \frac{\binom{n}{m}}{\binom{2n}{2m}}\frac{n!}{n^{2n-2m}} \sum_{p=m+2}^{2m+3} B^{p} \frac{ (n-p+1)^{2(n-p+1)}}{(n-p+1)! (p-1)!}.
\end{align*}
Note that the ratio between the $(p+1)$st term and the $p$th one is 
\begin{align*}
B &\cdot \frac { (n-p)^{2(n-p)} }{ (n-p+1)^{2(n-p+1)} } \cdot \frac { (n-p+1)! }{ (n-p)! } \cdot \frac {p-1}{p}\\
& = B \cdot \left( 1 - \frac {1}{n-p+1} \right)^{2(n-p+1)} \cdot \frac {n-p+1}{(n-p)^2}  \cdot \frac {p-1}{p} = \Theta (1/n).
\end{align*}
Hence,
\begin{align}
\Pr(\cC_m) &\le 4 \frac{\binom{n}{m}}{\binom{2n}{2m}}\frac{n! B^{m+2}}{n^{2n-2m}} \frac{ (n-m-1)^{2(n-m-1)}}{(n-m-1)! (m+1)!}. \label{eq:bound_cm}
\end{align}
If $m$ is a constant, then $\Pr(\cC_m) = O(1/n)$. In order to investigate larger values of $m$, note that the ratio between the $(m+1)$st term and the $m$th one is
\begin{align*}
\frac {\binom{n}{m+1}} {\binom{n}{m}} & \cdot \frac{\binom{2n}{2m}} {\binom{2n}{2m+2}} \cdot B n^2 \cdot \frac {(n-m-2)^{2(n-m-2)}}{(n-m-1)^{2(n-m-1)}} \cdot \frac {n-m-1}{m+2} \\
&= \frac {n-m}{m+1} \cdot \frac {(2m+1)(2m+2)}{(2n-2m)(2n-2m-1)} \cdot B n^2 \cdot \left( 1 - \frac {1}{n-m-1} \right)^{2(n-m-1)} \\
& \qquad \cdot \frac {n-m-1}{(n-m-2)^2} \cdot \frac {1}{m+2} \to (B)(e^{-2}) < 1 - 10^{-13} \qquad \text { as } m,n \to \infty.
\end{align*}
Hence, if $m$ is sufficiently large (say, $m \ge m'$) the ratio is at most, say, $1 - 10^{-14}$. Combining the two observations together, we get that 
$$
\sum_{1 \le m \le m_0} \Pr(\cC_m) = \sum_{1 \le m < m'} \Pr(\cC_m) + \sum_{m' \le m \le m_0} \Pr(\cC_m) = O \Big( \Pr(\cC_{1}) + \Pr(\cC_{m'}) \Big) = O \left( \frac {1}{n} \right) = o(1).
$$
That finishes the proof of Theorem~\ref{th1}.

%%%%%%%%%%%%%%%%%%%%%%%%%%%%
\section{Matchings and Hamilton Cycles}
%%%%%%%%%%%%%%%%%%%%%%%%%%%%

In this paper, we dealt with rainbow spanning trees proving the strongest possible result, both in terms of $q$, the number of colours, and $k$, the degree of the associated random graph. We leave investigating other rainbow structures for future research. 

Recall that it was shown by Frieze~\cite{F1} that $G_{2-out}$ has a perfect matching w.h.p., and by Bohman and Frieze~\cite{BF} that $G_{3-out}$ is Hamiltonian w.h.p. Both results are sharp. Hence, based on our observation in Section~\ref{sec:monotonicity}, it is natural to investigate the following questions.
\begin{itemize}
\item What is the smallest value of $q$ such that $G_{2,q}$ has a Rainbow Perfect Matching (\RPM) w.h.p.? (Trivially, $q \ge n/2$ and $q \le 2n$ as $G_{2,2n}$ is rainbow.)
\item What is the smallest value of $q$ such that $G_{3,q}$ has a Rainbow Hamilton Cycle (\RHC) w.h.p.? (Trivially, $q \ge n$ and $q \le 3n$ as $G_{3,3n}$ is rainbow.)
\end{itemize}

\end{document}